\documentclass{amsart}

\usepackage{amsmath}
\usepackage{amssymb}
\usepackage{amsthm}
\usepackage{graphicx}
\usepackage{bbm}

\newtheorem{theorem}{Theorem}

\newtheorem{lemma}[theorem]{Lemma}

\newcommand{\Oh}{\mathrm{O}}
\newcommand{\oh}{\mathrm{o}}
\newcommand{\im}{\mathrm{i}}
\newcommand{\e}{\mathrm{e}}
\newcommand{\dd}{\mathrm{d}}
\newcommand{\prob}{\mathbf{P}}

\title[The Hartman-Watson Distribution revisited]{The Hartman-Watson
Distribution revisited: Asymptotics for Pricing Asian Options}

\author{Stefan Gerhold}

\address{Vienna University of Technology, Wiedner Hauptstra\ss{}e 8--10,
A-1040 Vienna, Austria}

\email{sgerhold at fam.tuwien.ac.at}

\date{\today}

\begin{document}

\begin{abstract}
  Barrieu, Rouault, and Yor [\emph{J.~Appl.\ Probab.}~41 (2004)] determined
  asymptotics for the logarithm of the distribution function of the
  Hartman-Watson distribution. We determine the asymptotics
  of the density. This refinement can be applied to the pricing of
  Asian options in the Black-Scholes model.
\end{abstract}

\keywords{Hartman-Watson distribution, Asian option, saddle point method}

\subjclass[2010]{Primary: 62E20; Secondary: 60J65} 

\maketitle

\section{Introduction and main result}

The distribution of the integral of geometric Brownian motion has attracted
a lot of interest, in particular because it is needed to calculate the price
of Asian options in the Black-Scholes model.
Yor~\cite{Yo92} found the formula
\begin{equation}\label{eq:yor}
  \prob[A_t^{(\nu)} \in \dd u \mid W_t+\nu t=x] = \frac{\sqrt{2\pi t}}{u}
    \exp\left(\frac{x^2}{2t}-\frac{1+\e^{2x}}{2u} \right) \mathrm{I}_0(\e^x/u) f_{\e^x/u}(t)  
    \, \dd u,
\end{equation}
where $\mathrm{I}_{\nu}$ denotes, as usual, the modified Bessel function of the first
kind, and
\[
  A_t^{(\nu)} = \int_0^t \exp(2( W_h+\nu h))\, \dd h,
\]
where~$W$ is a standard Brownian motion. The present note focuses on the
function~$f_r(t)$ in~\eqref{eq:yor}, which is the density of the Hartman-Watson
distribution~\cite{HaWa74,Yo80}.
It is defined for a positive parameter~$r$ by the Laplace transform
\[
  \int_0^\infty \e^{-u t} f_r(t)\, \dd t = \frac{\mathrm{I}_{\sqrt{2u}}(r)}{\mathrm{I}_0(r)},
   \qquad \Re(u) > 0.
\]
Small time asymptotics of the conditional density~\eqref{eq:yor}
correspond to left tail asymptotics of~$f_r(t)$.
Numerical problems in the evaluation of~\eqref{eq:yor} for small~$t$ prompted
Barrieu, Rouault, and Yor~\cite{BaRoYo04} to analyze the left tail of the Hartman-Watson distribution asymptotically.
Using the G\"artner-Ellis theorem from large deviations theory, they obtained the asymptotics
\begin{equation}\label{eq:yor asympt}
  F_r(t) = \exp\left( - \frac{\log(1/t)^2}{2t} + \oh\Big(\frac{\log(1/t)^2}{t}\Big) \right), \qquad t\to 0,
\end{equation}
for the distribution function. However, this result is not immediately applicable
to the calculation of~\eqref{eq:yor}.
Barrieu et al.~\cite{BaRoYo04} write that
``the standard asymptotic methods (e.g.\ the saddle point method) do not seem to
be suitable for this study'', and that they ``are not able to refine these
results for the Hartman-Watson density itself''.

In fact the saddle point method~\cite{deBr58,Wo89} is applicable to
the Laplace inversion integral
\[
  f_r(t) = \frac{1}{2\im \pi} \int_{-\im\infty}^{+\im\infty} \e^{ut}
    \frac{\mathrm{I}_{\sqrt{2u}}(r)}{\mathrm{I}_0(r)}\, \dd u
\]
for the density~$f_r(t)$, but needs some care.
First, replacing $\mathrm{I}_{\sqrt{2u}}(r)$ by an asymptotic
approximation relieves us from studying potentially difficult monotonicity properties
of the modified Bessel function, and allows to formulate the result
in a way that avoids roots of equations involving the Bessel function.
Second, it turns out that elementary approximations of the integrand's
saddle point lead to integration
contours that are too far away from the saddle to make the method work. An approach
based on a contour through the \emph{exact} saddle point
establishes the following asymptotics for the density~$f_r(t)$.
For brevity, we write
\[
  \rho = \log \frac{r}{2\sqrt{2}}.
\]
\begin{theorem}\label{thm:main}
  For $t>0$, denote by $u_0(t)$ the largest solution of the equation
  \begin{equation}\label{eq:sp eq}
    t = \frac{\log u}{2 \sqrt{2u}} - \frac{\rho}{\sqrt{2u}}
    +\frac{1}{4u},
  \end{equation}
  which exists for all sufficiently small~$t$.
  Then the Hartman-Watson density satisfies
  \begin{align}
    f_r(t) &= \frac{\sqrt{\e}}{\pi\, \mathrm{I}_0(r)} \sqrt{\frac{u_0(t)}
      {\log u_0(t) - 2-2\rho}}\times \e^{-t u_0(t) 
    + \sqrt{2u_0(t)}} \left(1+\Oh(\sqrt{t}\log(1/t)^2)\right) \label{eq:main} \\
    &= \frac{\sqrt{\e}}{2\pi\, \mathrm{I}_0(r)} \frac{\log(1/t)^{1/2}}{t} \e^{-t u_0(t) 
    + \sqrt{2u_0(t)}} \left(1+\Oh\Big(\frac{\log \log (1/t)}{\log(1/t)}\Big)\right) \label{eq:main rough}
  \end{align}
  as $t\to0$.
\end{theorem}
Formula~\eqref{eq:main} gives a much better approximation than~\eqref{eq:main rough};
the simplification
in~\eqref{eq:main rough} is of little use, since~$u_0(t)$ has to be computed
anyway to evaluate~\eqref{eq:main} or~\eqref{eq:main rough} numerically.

To get a feel for the growth of the exponential in~\eqref{eq:main}, we expand~$u_0(t)$
by bootstrapping (cf.\ de~Bruijn~\cite[Section~2.4]{deBr58}):
\begin{equation}\label{eq:u0 expans}
  u_0(t) = \frac{\log(1/t)^2}{2t^2}\left( 1 + \frac{2\log \log(1/t)}{\log(1/t)}
    - \frac{2\rho+\log 2}{\log(1/t)} + \oh\Big(\frac{1}{\log(1/t)}\Big) \right).
\end{equation}
Therefore the exponent in~\eqref{eq:main} has the expansion
\begin{multline}
  -t u_0(t) + \sqrt{2u_0(t)} 
    = -\frac{\log(1/t)^2}{2t}
    - \frac{\log(1/t)\log\log(1/t)}{t} \\
  + (1+\rho+\tfrac12 \log2)\frac{\log(1/t)}{t} + \oh\Big(\frac{\log(1/t)}{t}\Big).
  \label{eq:exp expans}
\end{multline}
This shows in particular that the formula
\begin{equation}\label{eq:f rough}
  f_r(t) = \exp\left( - \frac{\log(1/t)^2}{2t} + \oh\Big(\frac{\log(1/t)^2}{t}\Big) \right),
\end{equation}
obtained from~\eqref{eq:yor asympt} by formal differentiation, is correct.

For numerical accuracy, it is certainly preferable to use~\eqref{eq:main}
as it is, without replacing the exponent by~\eqref{eq:exp expans};
still, the expansion~\eqref{eq:u0 expans} can serve as good initial
guess when computing the root of~\eqref{eq:sp eq}. In this way, the leading term
of~$f_r(t)$ can be calculated effortlessly even for extremely small values
of~$t$, say $t=10^{-50}$.

\section{Analysis of the Laplace inversion integral}

The Laplace inversion formula yields the representation
\[
  f_r(t) = \frac{1}{2\im \pi} \int_{R-\im\infty}^{R+\im\infty} \e^{ut}
    \frac{\mathrm{I}_{\sqrt{2u}}(r)}{\mathrm{I}_0(r)}\, \dd u,
\]
where $R>0$, so that the integration contour lies in the right half-plane. To estimate the
growth of~$f_r(t)$ near $t=0$, we have to investigate the singularity at infinity
of the integrand.
For large index, the modified Bessel function admits the expansion~\cite{Ho99, Wa95}
\begin{equation}\label{eq:I asympt}
  \mathrm{I}_{\nu}(r) \sim \left( \frac{r}{2} \right)^\nu \e^\nu \nu^{-\nu-1/2}
  \left(c_0 + \frac{c_1}{\nu} + \dots \right),
\end{equation}
where the~$c_i$ are constants, with $c_0=1/\sqrt{2\pi}$.
This holds for $\nu\to\infty$, uniformly w.r.t.\ $\arg(\nu)$, as long
as $\arg(\nu)$ is bounded away from~$\pm\pi$.
Horn~\cite{Ho99} shows~\eqref{eq:I asympt} for~$\mathrm{J}_{\nu}(r)$,
the Bessel function of the first kind, but from the relation $\mathrm{I}_{\nu}(r)
=\e^{-\nu\pi\im/2} \mathrm{J}_{\nu}(r \e^{\pi\im/2})$ one easily sees that
replacing~$\mathrm{J}_{\nu}(r)$ by~$\mathrm{I}_{\nu}(r)$ only affects the
constants $c_1,c_2,\dots$ in this expansion.
If we let the real part~$R$ of our integration contour tend to infinity
as $t\to0$, we therefore have
\begin{align*}
  f_r(t) &=  \frac{1}{(2 \pi)^{3/2}\, \im\, \mathrm{I}_0(r)} \int_{R-\im\infty}^{R+\im\infty}
  \left( \frac{r}{2} \right)^{\sqrt{2u}} \e^{ut+\sqrt{2u}} (\sqrt{2u})^{-\sqrt{2u}-1/2}
  \dd u \times (1+\Oh(R^{-1/2})) \\
  &= \frac{2^{-7/4}}{\pi^{3/2}\, \im\, \mathrm{I}_0(r)} \int_{R-\im\infty}^{R+\im\infty}
  \exp\Big(ut - \tfrac12 \sqrt{2u}\log u + \sqrt{2}(1+\rho)\sqrt{u}  \\
  & \qquad \qquad -\tfrac14 \log u \Big)\, \dd u \times (1+\Oh(R^{-1/2})).
\end{align*}
The integrand of the latter integral has a saddle point, let us call
it $u_0=u_0(t)$, which is found by equating the derivative to zero.
This yields equation~\eqref{eq:sp eq}.
Shifting the integration contour through the saddle point achieves concentration,
so that only a small part of the contour matters asymptotically.

In many instances of the saddle point
method, it suffices to choose a contour that passes through an approximation
of the saddle point.
In our example, one might try to use a contour based on
the first terms of~\eqref{eq:u0 expans}.
However, painful calculations reveal that the concentration of the integrand
around the approximate saddle point is insufficient, no matter how many
terms of~\eqref{eq:u0 expans} are taken.
We therefore set the real part
of the integration contour to the \emph{exact} saddle point, so that $R=u_0$:
\begin{equation}\label{eq:simple int}
  f_r(t) \sim \frac{2^{-7/4}}{\pi^{3/2}\, \im\, \mathrm{I}_0(r)}  
    \int_{u_0-\im\infty}^{u_0+\im\infty}
  \exp\Big(ut - \tfrac12 \sqrt{2u} \log u + \sqrt{2}(1+\rho)\sqrt{u}
  -\tfrac14 \log u \Big)\, \dd u.
\end{equation}
Let~$y$ denote the new (real) integration variable:
\[
  u = u_0 + \im y, \qquad -\infty < y < \infty.
\]
Close to the saddle point, i.e.\ for small values of the new integration
variable~$y$, we have the uniform expansions
\begin{align*}
  \sqrt{u} &= \sqrt{u_0} + \frac{\im y}{2\sqrt{u_0}} + \frac{y^2}{8u_0^{3/2}}
    + \Oh\Big(\frac{y^3}{u_0^{3/2}}\Big) , \\
  \log u &= \log u_0 + \frac{\im y}{u_0} + \frac{y^2}{2u_0^2} + \Oh\Big(\frac{y^3}{u_0^3}\Big),
\end{align*}
and
\[
  \sqrt{u}\log u = \sqrt{u_0}\log u_0+ \frac{\im y}{\sqrt{u_0}}
    + \frac{\im \log(u_0) y}{2\sqrt{u_0}} + \frac{\log(u_0)y^2}{8u_0^{3/2}}
    + \Oh\Big(\frac{\log(u_0)y^3}{u_0^{5/2}}\Big).
\]
We insert these into the exponent of~\eqref{eq:simple int} and obtain
\begin{align}
  ut  - \tfrac12 \sqrt{2u}\log u & +   
    \sqrt{2}(1+\rho)\sqrt{u} 
    -\tfrac14 \log u \notag \\
  &= u_0 t -\tfrac12 \sqrt{2u_0} \log u_0
    + \sqrt{2}(1+\rho)\sqrt{u_0} - \tfrac14 \log u_0 \notag \\
  &\quad -My^2 + \Oh\Big(\frac{\log(u_0)y^3}{u_0^{5/2}}\Big), \label{eq:loc expans}
\end{align}
where
\begin{align}
  M &= \frac{\sqrt{2} \log u_0}{16u_0^{3/2}} - \frac{\sqrt{2}(1+\rho)}{8u_0^{3/2}}
  \label{eq:M} \\
    &= \frac{t^3}{2 \log(1/t)^2} \left(1+\Oh\Big(\frac{\log\log(1/t)}{\log(1/t)}\Big)\right).
    \label{eq:M expans}
\end{align}
Note that the $y$-terms in~\eqref{eq:loc expans} vanish, because
we integrate through a saddle point.
We now have to identify a range
\[
  -h < y < h
\]
for~$y=\Im(u)$ that captures the main contribution to the integral~\eqref{eq:simple int}.
A good choice is
\[
  h = \frac{\log(1/t)^2}{t^{3/2}},
\]
because it satisfies $h\sqrt{M}\to\infty$, so that the integral of the local
expansion~\eqref{eq:loc expans} can be completed to a full Gaussian integral:
\begin{align}
  \int_{-h}^h \e^{-M y^2} \dd y &=
    \frac{1}{\sqrt{2M}} \int_{-h\sqrt{2M}}^{h\sqrt{2M}} \e^{-w^2/2} \dd w \notag \\
    &\sim \frac{1}{\sqrt{2M}} \int_{-\infty}^\infty \e^{-w^2/2} \dd w \label{eq:gauss} \\
    &= \sqrt{\frac{\pi}{M}} \sim \frac{\sqrt{2\pi} \log(1/t)}{t^{3/2}}. \notag
\end{align}
Moreover, the error from~\eqref{eq:loc expans}, the local expansion at
the saddle point, is~$\oh(1)$, since
\begin{equation}\label{eq:err loc}
  \frac{\log(u_0)y^3}{u_0^{5/2}} = \Oh(\sqrt{t} \log(1/t)^2).
\end{equation}
We can thus determine the asymptotics of the portion $|\Im(u)|\leq h$
of the integral~\eqref{eq:simple int}:
\begin{align}
  \frac{2^{-7/4}}{\pi^{3/2}\, \im\, \mathrm{I}_0(r)} & \int_{u_0-\im h}^{u_0+\im h}
  \exp\Big(ut - \tfrac12 \sqrt{2u}\log u + \sqrt{2}(1+\rho)\sqrt{u}
    -\tfrac14 \log u \Big)\, \dd u \label{eq:central} \\
    &\sim \frac{2^{-7/4}}{\pi^{3/2}\, \mathrm{I}_0(r)}
      \e^{u_0 t -\tfrac12 \sqrt{2u_0} \log u_0
    + \sqrt{2}(1+\rho)\sqrt{u_0} - \tfrac14 \log u_0}\int_{-h}^h \e^{-My^2}\dd y \notag \\
    &\sim \frac{2^{-7/4}}{\pi\, \mathrm{I}_0(r)} M^{-1/2} u_0^{-1/4}
    \e^{u_0 t -\tfrac12 \sqrt{2u_0} \log u_0 + \sqrt{2}(1+\rho)\sqrt{u_0}}.
      \label{eq:central2}
\end{align}
This gives the right-hand side of~\eqref{eq:main}, after
expressing $\sqrt{u_0}\log u_0$ via the saddle point equation~\eqref{eq:sp eq},
which yields
\begin{equation}\label{eq:from sp}
  -\tfrac12 \sqrt{2u_0}\log u_0 = -2u_0 t-\rho \sqrt{2u_0}+\tfrac12.
\end{equation}
Furthermore, expanding~$u_0$ by~\eqref{eq:u0 expans}
gives the expression in~\eqref{eq:main rough}.
Note that we have not yet proved~\eqref{eq:main} and~\eqref{eq:main rough};
it remains to show that the tails of~\eqref{eq:simple int}, i.e.\
the parts where $|\Im(u)| \geq h$, are asymptotically negligible.
This ensures that~\eqref{eq:central} indeed captures the asymptotics of~$f_r(t)$.

\section{Tail estimate}

To bound the tails of~\eqref{eq:simple int},
it suffices to consider the case $y=\Im(u)\geq h$,
since the lower half of the tail can be handled by symmetry.
We first deal with the part of the contour in~\eqref{eq:simple int} where the
imaginary part of the integration variable is very large, say
$y\geq \e^{\log(1/t)^2/4}$. Then~$y$ clearly dominates~$u_0$,
and it follows from
\[
  \Re(\log u) \sim \log y, \quad \Re(\sqrt{u}) \sim \tfrac12 \sqrt{2y},
  \quad \text{and}\quad \Re(\sqrt{u}\log u) \sim \tfrac12 \sqrt{2y} \log y
\]
that the absolute value of the integrand is bounded by
\[
  \e^{u_0t - \sqrt{y}}
\]
for small~$t$. Hence we obtain the bound
\begin{equation}\label{eq:outer tail}
  \e^{u_0 t} \int_{\e^{\log(1/t)^2/4}}^\infty \e^{-\sqrt{y}} \dd y \sim
    2\exp\left( u_0t + \tfrac18 \log(1/t)^2 -\e^{\log(1/t)^2/8} \right).
\end{equation}

Finally, we bound the portion of the integral~\eqref{eq:simple int}
that is close to the central part, i.e.,
\begin{equation}\label{eq:y inner}
  h \leq y < \e^{\log(1/t)^2/4}.
\end{equation}
The following lemma shows that, for small~$t$, the absolute value of the integrand decreases
as~$y$ increases.
\begin{lemma}
  Let~$B$ be a real number. Then,
  for $\Re(u)>0$ and $|u|$ sufficiently large, the real part of
  $\sqrt{u}\log u + B \sqrt{u}$
  decreases w.r.t. $|\Im(u)|$.
\end{lemma}
\begin{proof}
  We write $u=x + \im y$. By symmetry, it suffices to consider the case $y>0$,
  so that $\arg(u)>0$.
  Straightforward calculations show that
  \[
    \frac{\partial}{\partial y} \Re(\sqrt{u})
    = \tfrac12 |u|^{3/2}\left( y \cos \tfrac{\arg(u)}{2}
    - x \sin \tfrac{\arg(u)}{2} \right)
  \]
  and
  \begin{align*}
    \frac{\partial}{\partial y} \Re(\sqrt{u}\log u)
    = &\tfrac12 |u|^{3/2}\bigg( \left(\log|u| +2\right)\left(y \cos \tfrac{\arg(u)}{2}
    - x \sin \tfrac{\arg(u)}{2} \right) \\
    &- \arg(u)\left(x \cos \tfrac{\arg(u)}{2}
    + y \sin \tfrac{\arg(u)}{2}\right) \bigg).
  \end{align*}
  Hence we are led to investigate the sign of
  \[
    y\left(\left(\log|u| +B+2\right)\left( \cos \tfrac{\arg(u)}{2} - \tfrac{x}{y} \sin
    \tfrac{\arg(u)}{2} \right) -  \arg(u) \left(\tfrac{x}{y}\cos \tfrac{\arg(u)}{2} +  \sin  
    \tfrac{\arg(u)}{2}\right) \right).
  \]
  Suppose that~$|u|$ is so large that $\log|u| +B+2 \geq 12$. In the preceding formula,
  we estimate the trigonometric functions by the first
  term of their Taylor series at zero, except the first cos, where we use two terms.
  This yields the lower bound
  \begin{align*}
    & y\left(12\left(1-\tfrac18 \arg(u)^2-\tfrac{x}{2y} \arg(u)\right)
    -\tfrac{x}{y}\arg(u) - \tfrac12 \arg(u)^2 \right) \\
    & =y\left(12-2\arg(u)^2 - \tfrac{7x}{y} \arg(u)\right) \\
    & =y\left.\left(12-2\arctan(w)^2-\tfrac{7}{w}\arctan(w) \right)\right|_{w=y/x}.
  \end{align*}
  Now observe that $\arctan(w)^2 < \pi^2/4$ and $\arctan(w)<w$ for
  $w>0$, so that
  \[
    12-2\arctan(w)^2-\tfrac{7}{w}\arctan(w) > 12 - \pi^2/2 - 7 > 0.
  \]
  This shows that $\Re(\sqrt{u}\log u + B \sqrt{u})$ has a positive derivative
  w.r.t.~$y$.
\end{proof}
Therefore, we can bound the part~\eqref{eq:y inner} of the integral~\eqref{eq:simple int}
by the value of the integrand
at~$y=h$ times the length of the path. By~\eqref{eq:loc expans},
\eqref{eq:from sp}, and
\[
  \left.M y^2\right|_{y=h} \sim \tfrac12 \log(1/t)^2,
\]
this amounts to a bound of the form
\begin{multline}\label{eq:inner tail}
  \e^{-t u_0(t) 
    + \sqrt{2}(1+2\rho)\sqrt{u_0(t)} - \frac12 \log(1/t)^2 + \oh(\log(1/t)^2)}
    \times \e^{\frac14 \log(1/t)^2} \\
    = \e^{-t u_0(t) 
    + \sqrt{2}(1+2\rho)\sqrt{u_0(t)}-\frac14 \log(1/t)^2 + \oh(\log(1/t)^2)}.
\end{multline}
To complete the proof of Theorem~\ref{thm:main},
let us now compare the six error terms that arose in the analysis (see
Table~\ref{ta:err}).
Note that the error from completing the tails of the Gaussian integral
in~\eqref{eq:gauss} is
\begin{align}
  \sqrt{\frac2M}\int_{h\sqrt{2M}}^\infty \e^{-w^2/2} \dd w &\sim
    \sqrt{\frac2M}\left.\frac{\e^{-w^2/2}}{w}\right|_{w=h\sqrt{2M}} \notag \\
    &= \e^{-\frac12 \log(1/t)^2 + \oh(\log(1/t)^2)}. \label{eq:gauss err}
\end{align}
\begin{table}[h]
  \begin{center}
    \begin{tabular}{l|l} 
       \emph{Source of error} & \emph{Relative error} \\ \hline\hline
       Replace $I_\nu$ by~\eqref{eq:I asympt}  & $\Oh(t/\log(1/t))
       \vphantom{X^{X^{X}}}$  \\ \hline
       Local expansion (see~\eqref{eq:loc expans}
         and~\eqref{eq:err loc}) & $\Oh(\sqrt{t}\log(1/t)^2)
         \vphantom{X^{X^{X^{X}}}}$ \\ \hline
       Gaussian tails (see~\eqref{eq:gauss} and~\eqref{eq:gauss err}) &
         $\exp(-\tfrac12 \log(1/t)^2
         + \oh(\log(1/t)^2)) \vphantom{X^{X^{X^{X}}}}$ \\ \hline
       Relative error of~$M$ (see~\eqref{eq:M expans} and~\eqref{eq:gauss}) &
       $\Oh(\frac{\log\log(1/t)}{\log(1/t)}) \vphantom{X^{X^{X^{X}}}}$ \\ \hline
       Outer tail (see~\eqref{eq:outer tail}) & 
       $\exp(-\e^{\log(1/t)^2/8} + \oh(\e^{\log(1/t)^2/8})) \vphantom{X^{X^{X^{X}}}}$ \\ \hline
       Inner tail (see~\eqref{eq:inner tail}) &
         $\exp(-\tfrac14 \log(1/t)^2 + \oh(\log(1/t)^2)) \vphantom{X^{X^{X^{X}}}}$
    \end{tabular}
  \end{center}
  \caption{}
  \label{ta:err}
\end{table}
If~$M$ is not expanded, i.e., \eqref{eq:M} is used, then
the error from the local expansion dominates, which leads to~\eqref{eq:main}.
If, on the other hand, the expansion~\eqref{eq:M expans} of~$M$ is taken,
then it is the relative error of~$M$ that prevails.

\section{Comments}

The left tail of the Hartman-Watson distribution (see~\eqref{eq:f rough})
is somewhat thinner than that
of the L\'{e}vy distribution (stable distribution with index $\alpha=\tfrac12$),
with density
\[
  g(t) = \frac{1}{\sqrt{2\pi t^3}}\, \e^{-1/2t}, \qquad t>0,
\]
and Laplace transform
\[
  \int_0^\infty \e^{-ut}g(t)\, \dd t = \e^{-\sqrt{2u}}, \qquad \Re(u)>0.
\]
The faster decay of the Laplace transform of the Hartman-Watson distribution,
of order $\exp(-\sqrt{u}\log u)$, becomes manifest in the additional factor
$\log(1/t)^2$ in the exponent of~\eqref{eq:f rough}.

We now briefly comment on possible refinements of
Theorem~\ref{thm:main}. Technically speaking, continuing the
expansion~\eqref{eq:M expans} and inserting into~\eqref{eq:central2}
refines~\eqref{eq:main rough} to a full asymptotic expansion.
A better expansion, respecting the asymptotic scale of the problem, can be
obtained by retaining the explicit formula~\eqref{eq:M} for~$M$, and
taking more terms in~\eqref{eq:I asympt} and~\eqref{eq:loc expans}.
This should pose no essential difficulties; note, however, that each term
in the expansion~\eqref{eq:I asympt} gives rise to a new saddle point,
as the coefficient of~$1/u$ in~\eqref{eq:sp eq} changes. Thus the expansion
will involve several implicitly defined functions of~$t$ besides~$u_0(t)$.
Since the dependence of the solution of~\eqref{eq:sp eq} on the coefficient
of~$1/u$ is light, it might be possible to give an expansion that features
only~$u_0(t)$. This seems of little practical interest, though.

Concerning applications,
our results can be used as a substitute for the Hartman-Watson density for small
arguments, in particular, for evaluating the density of~$A_t^{(\nu)}$ numerically
for small~$t$
(after integrating~\eqref{eq:yor} w.r.t.\ the law of $W_t+\nu t$).
The related problem of
determining small time asymptotics for the density of~$A_t^{(\nu)}$ is left to future research.
This density is difficult to evaluate numerically for small time~\cite{Is05}.
Analyzing it asymptotically requires handling a double integral;
see Tolmatz~\cite{To00,To05} for asymptotic evaluations of double integrals pertaining
to other functionals of Brownian motion (resp.\ the Brownian bridge).
For $\nu=0$, the analysis should be simpler, as the density of~$A^{(0)}_t$
can be expressed as a single integral, via Bougerol's identity~\cite{Fo11,MaYo05}.
Note that tail asymptotics~\cite{GuSt06,GuSt10} and large time
asymptotics~\cite{Fo11,MaYo05} of~$A_t^{(\nu)}$ are known.
See also Dufresne~\cite{Du04} for related limit laws.

\bibliographystyle{siam}
\bibliography{../gerhold}

\begin{thebibliography}{10}

\bibitem{BaRoYo04}
{\sc P.~Barrieu, A.~Rouault, and M.~Yor}, {\em A study of the
  {H}artman-{W}atson distribution motivated by numerical problems related to
  the pricing of {A}sian options}, J. Appl. Probab., 41 (2004), pp.~1049--1058.

\bibitem{deBr58}
{\sc N.~G. de~Bruijn}, {\em Asymptotic methods in analysis}, Bibliotheca
  Mathematica. Vol. 4, North-Holland Publishing Co., Amsterdam, 1958.

\bibitem{Du04}
{\sc D.~Dufresne}, {\em The log-normal approximation in financial and other
  computations}, Adv. in Appl. Probab., 36 (2004), pp.~747--773.

\bibitem{Fo11}
{\sc M.~Forde}, {\em Exact pricing and large-time asymptotics for the modified
  {SABR} model and the {B}rownian exponential functional}.
\newblock To appear in International Journal of Theoretical and Applied
  Finance, 2011.

\bibitem{GuSt06}
{\sc A.~Gulisashvili and E.~M. Stein}, {\em Asymptotic behavior of the
  distribution of the stock price in models with stochastic volatility: the
  {H}ull-{W}hite model}, C. R. Math. Acad. Sci. Paris, 343 (2006),
  pp.~519--523.

\bibitem{GuSt10}
\leavevmode\vrule height 2pt depth -1.6pt width 23pt, {\em Asymptotic behavior
  of distribution densities in models with stochastic volatility. {I}}, Math.
  Finance, 20 (2010), pp.~447--477.

\bibitem{HaWa74}
{\sc P.~Hartman and G.~S. Watson}, {\em ``{N}ormal'' distribution functions on
  spheres and the modified {B}essel functions}, Ann. Probability, 2 (1974),
  pp.~593--607.

\bibitem{Ho99}
{\sc J.~Horn}, {\em Ueber lineare {D}ifferentialgleichungen mit einem
  ver\"anderlichen {P}arameter}, Math. Ann., 52 (1899), pp.~340--362.

\bibitem{Is05}
{\sc K.~Ishiyama}, {\em Methods for evaluating density functions of exponential
  functionals represented as integrals of geometric {B}rownian motion},
  Methodol. Comput. Appl. Probab., 7 (2005), pp.~271--283.

\bibitem{MaYo05}
{\sc H.~Matsumoto and M.~Yor}, {\em Exponential functionals of {B}rownian
  motion. {I}. {P}robability laws at fixed time}, Probab. Surv., 2 (2005),
  pp.~312--347 (electronic).

\bibitem{To00}
{\sc L.~Tolmatz}, {\em Asymptotics of the distribution of the integral of the
  absolute value of the {B}rownian bridge for large arguments}, Ann. Probab.,
  28 (2000), pp.~132--139.

\bibitem{To05}
\leavevmode\vrule height 2pt depth -1.6pt width 23pt, {\em Asymptotics of the
  distribution of the integral of the positive part of the {B}rownian bridge
  for large arguments}, J. Math. Anal. Appl., 304 (2005), pp.~668--682.

\bibitem{Wa95}
{\sc G.~N. Watson}, {\em A treatise on the theory of {B}essel functions},
  Cambridge Mathematical Library, Cambridge University Press, Cambridge, 1995.
\newblock Reprint of the second (1944) edition.

\bibitem{Wo89}
{\sc R.~Wong}, {\em Asymptotic approximations of integrals}, Computer Science
  and Scientific Computing, Academic Press Inc., Boston, MA, 1989.

\bibitem{Yo80}
{\sc M.~Yor}, {\em Loi de l'indice du lacet brownien, et distribution de
  {H}artman-{W}atson}, Z. Wahrsch. Verw. Gebiete, 53 (1980), pp.~71--95.

\bibitem{Yo92}
\leavevmode\vrule height 2pt depth -1.6pt width 23pt, {\em On some exponential
  functionals of {B}rownian motion}, Adv. in Appl. Probab., 24 (1992),
  pp.~509--531.

\end{thebibliography}

\end{document}